\documentclass{lmsstyle}
\usepackage{amsmath}
\usepackage{amsfonts}
\usepackage{graphicx}

\newtheorem{theorem}{Theorem}[section] 
\newtheorem{lemma}[theorem]{Lemma}     
\newtheorem{corollary}[theorem]{Corollary}
\newtheorem{proposition}[theorem]{Proposition}

\providecommand{\set}[1]{\left\{#1\right\}}

\DeclareMathOperator{\gen}{gen}

\title[Box-counting dimension product formulas]
 {Strict inequality in the box-counting dimension product formulas} 

\author{N. Sharples}

\classno{28A75 (primary), 28A80 (secondary).}

\extraline{This research was funded by an EPSRC Research Studentship.}

\begin{document}
\maketitle

\begin{abstract}
It is known that the upper box-counting dimension of a Cartesian product satisfies the inequality $\dim_{B}\left(F\times G\right)\leq \dim_{B}\left(F\right) + \dim_{B}\left(G\right)$ whilst the lower box-counting dimension satisfies the inequality $\dim_{LB}\left(F\times G\right)\geq \dim_{LB}\left(F\right) + \dim_{LB}\left(G\right)$. We construct Cantor-like sets to demonstrate that both of these inequalities can be strict.
\end{abstract}


\section{Preliminaries} 
\label{intro}
\noindent
In a metric space $X$ the Hausdorff dimension of a compact set $F\subset X$ is defined as the supremum of the $d\geq 0$ such that the $d$-dimensional Hausdorff measure $\mathcal{H}^{d}\left(F\right)$ of $F$ is infinite. The Hausdorff dimension takes values in the non-negative reals and extends the elementary integer-valued topological dimension in the sense that for a large class of `reasonable' sets these two values coincide. Sets with non-coinciding Hausdorff and topological dimensions are called `fractal', a term coined by Mandelbrot in his original study of such sets \cite{Mandelbrot75}. Hausdorff introduced this generalised dimension in \cite{Hausdorff18} and its subsequent extensive use in geometric measure theory is developed by Federer \cite{BkFederer69} and Falconer \cite{BkFalconer85}. The fact that the Hausdorff dimension satisfies $\dim_{H}\left(F\times G\right)\geq \dim_{H}\left(F\right) + \dim_{H}\left(G\right)$ for the Cartesian product of sets was proved in full generality in \cite{Marstrand53} (and later summarised in \cite{BkFalconer03} \S 7.1 `Product Formulae') after some partial results: The inequality was proved in \cite{BesicovitchMoran45} with the restriction that $0<\mathcal{H}^{s}\left(F\right),\mathcal{H}^{t}\left(G\right)<\infty$ for some $s,t$ and was extended to a larger class of sets in \cite{Eggleston53}. The paper \cite{BesicovitchMoran45} also provides an example for which there is a strict inequality in the product formula and again this is summarised in \cite{BkFalconer03} \S 7.1.
\paragraph*{}
In this paper we prove similar product inequalities for the upper and lower box-counting dimensions which are less familiar generalisations of dimension (treated briefly in \cite{BkFalconer03}; see \cite{BkRobinson10} for a more detailed exposition) and have applications to dynamical systems (see, for example, \cite{BkRobinson01}). Our main result is an example analogous to that in \cite{BesicovitchMoran45} which demonstrates that the box-counting product inequalities can be strict.
In a metric space $X$ the upper and lower box-counting dimensions of a compact set $F\subset X$ are defined by
\begin{align}
\dim_{B}\left(F\right) &= \limsup_{\delta\searrow 0}\frac{\log\left(N\left(F,\delta\right)\right)}{-\log\delta}\label{dimB def}
\intertext{and}
\dim_{LB}\left(F\right) &= \liminf_{\delta\searrow 0}\frac{\log\left(N\left(F,\delta\right)\right)}{-\log\delta}\label{dimLB def}
\end{align}
respectively, where $N\left(F,\delta\right)$ is the smallest number of sets with diameter at most $\delta$ which form a cover (called a $\delta$-cover) of $F$. Essentially, if $N\left(F,\delta\right)$ scales like $\delta^{-\varepsilon}$ as $\delta\rightarrow 0$ then these quantities capture $\varepsilon$ which gives an indication of how many more sets are required to cover $F$ as the length-scales decrease and so encodes how `spread out' the set $F$ is at small length-scales. These limits are unchanged if we replace $N\left(F,\delta\right)$ with one of many similar quantities (discussed by Falconer in \cite{BkFalconer03} \S 3.1 `Equivalent Definitions'). Of these quantities we will also make use of the largest number of disjoint closed balls of diameter $\delta$ with centres in $F$, which we denote $M\left(F,\delta\right)$.

\paragraph*{}
We first recall the standard (see, for example, \cite{BkFalconer03}  or \cite{BkRobinson10}) proof of the box-counting product inequalities when $F$ and $G$ are compact sets in metric spaces $X$ and $Y$ respectively, although the inequality \eqref{dimLB product inequality} is less familiar (Robinson provides a proof in \cite{BkRobinson10}). We endow the product space $X\times Y$ with the usual Euclidean metric $d_{X\times Y}=\sqrt{d_{X}^{2}+d_{Y}^{2}}$, but the proof can be adapted for a variety of product metrics (see \cite{BkRobinson10}).
\begin{theorem}\label{product inequalities theorem}
For compact sets $F\subset X$ and $G\subset Y$ the box-counting dimensions of the product set $F\times G$ satisfy the inequalities
\begin{align}
\dim_{B}\left(F\times G\right)&\leq \dim_{B}\left(F\right) + \dim_{B}\left(G\right)\label{dimB product inequality}\\
\dim_{LB}\left(F\times G\right)&\geq \dim_{LB}\left(F\right) + \dim_{LB}\left(G\right)\label{dimLB product inequality}
\end{align}
\end{theorem}
\begin{proof}
Suppose $\set{U_{i}}_{i=1}^{n_{1}}$ and $\set{V_{j}}_{j=1}^{n_{2}}$ are $\delta$-covers of $F$ and of $G$ respectively then the set of products $\set{U_{i}\times V_{j}\vert i=1\ldots n_{1}, j=1\ldots n_{2}}$ is a cover of $F\times G$ with a total of $n_{1}n_{2}$ elements and the diameter of each $U_{i}\times V_{j}$ is no greater than $\sqrt{2}\delta$. As there exist $\delta$-covers of $F$ and $G$ consisting of $N\left(F,\delta\right)$ and $N\left(G,\delta\right)$ elements respectively this construction gives a $\sqrt{2}\delta$-cover of $F\times G$ consisting of $N\left(F,\delta\right)N\left(G,\delta\right)$ elements, hence
\begin{equation}\label{minimal cover inequality}
N\left(F\times G,\sqrt{2}\delta\right) \leq N\left(F,\delta\right)N\left(G,\delta\right).
\end{equation}
Next, if both $\set{x_{i}}_{i=1}^{n_{1}}\subset F$ and $\set{y_{j}}_{j=1}^{n_{2}}\subset G$ are sets of centres of disjoint balls with diameter $\delta$ then the balls with radius $\delta$ centred on the $n_{1}n_{2}$ points $\set{\left(x_{i},y_{j}\right)\vert i=1\ldots n_{1}, j=1\ldots n_{2}}\subset F\times G$, are also disjoint. As there exist sets of disjoint balls of diameter $\delta$ with centres in $F$ and $G$ consisting of $M\left(F,\delta\right)$ and $M\left(G,\delta\right)$ elements respectively the above construction gives $M\left(F,\delta\right)M\left(G,\delta\right)$ disjoint balls of diameter $\delta$ with centres in $F\times G$, hence
\begin{equation}\label{maximal disjoint inequality}
M\left(F\times G,\delta\right) \geq M\left(F,\delta\right)M\left(G,\delta\right).
\end{equation}
From \eqref{dimB def} and \eqref{minimal cover inequality} we see that
\begin{align}
\dim_{B}\left(F\times G\right) &= \limsup_{\delta\searrow 0}\frac{\log\left(N\left(F\times G,\sqrt{2}\delta\right)\right)}{-\log\left(\sqrt{2}\delta\right)}\notag\\
&\leq \limsup_{\delta\searrow 0} \left[\frac{\log\left(N\left(F,\delta\right)\right)}{-\log\left(\sqrt{2}\delta\right)} + \frac{\log\left(N\left(G,\delta\right)\right)}{-\log\left(\sqrt{2}\delta\right)}\right]\label{limsup of sum}\\
&\leq \limsup_{\delta\searrow 0} \frac{\log\left(N\left(F,\delta\right)\right)}{-\log\delta -\log\sqrt{2}} + \limsup_{\delta\searrow 0} \frac{\log\left(N\left(G,\delta\right)\right)}{-\log\delta -\log\sqrt{2}}\notag\\
& = \dim_{B}\left(F\right) + \dim_{B}\left(G\right). \notag
\intertext{From \eqref{dimLB def} and \eqref{maximal disjoint inequality} we have}
\dim_{LB}\left(F\times G\right) &= \liminf_{\delta\searrow 0}\frac{\log\left(M\left(F\times G,\delta\right)\right)}{-\log\delta}\notag\\
&\geq \liminf_{\delta\searrow 0}\left[\frac{\log\left(M\left(F,\delta\right)\right)}{-\log\delta} + \frac{\log\left(M\left(G,\delta\right)\right)}{-\log\delta}\right]\label{liminf of sum}\\
&\geq \liminf_{\delta\searrow 0}\frac{\log\left(M\left(F,\delta\right)\right)}{-\log\delta} + \liminf_{\delta\searrow 0}\frac{\log\left(M\left(G,\delta\right)\right)}{-\log\delta} \notag \\
&= \dim_{LB}\left(F\right) + \dim_{LB}\left(G\right). \notag
\end{align}
\end{proof}
It is known that there are sets with unequal upper and lower box-counting dimension (see exercise 3.8 of \cite{BkFalconer03} or \cite{BkRobinson10} \S 3.1), however if these values coincide for a set $F$ we define their common value as the box-counting dimension of $F$. If sets $F$ and $G$ have well-defined box-counting dimension then the box-counting dimension of their product is also well behaved.

\begin{corollary}
If $\dim_{B}\left(F\right)=\dim_{LB}\left(F\right)$ and $\dim_{B}\left(G\right)=\dim_{LB}\left(G\right)$ then 
\[\dim_{B}\left(F\times G\right) = \dim_{LB}\left(F\times G\right) = \dim_{B}\left(F\right)+\dim_{B}\left(G\right).\]
\end{corollary}
\begin{proof}
From the inequalities \eqref{dimB product inequality} and \eqref{dimLB product inequality} we have
\[
\dim_{B}\left(F\times G\right) \leq \dim_{B}\left(F\right) + \dim_{B}\left(G\right) = \dim_{LB}\left(F\right) + \dim_{LB}\left(G\right) \leq \dim_{LB}\left(F\times G\right)
\]
but from the definition of the box-counting dimension $\dim_{LB}\left(F\times G\right)\leq \dim_{B}\left(F\times G\right)$ so we must have equality throughout the above.
\end{proof}

\paragraph*{}
In the following construction both sets $F$ and $G$ have non-coinciding upper and lower box-counting dimensions so that as $\delta\rightarrow 0$ the box-counting functions $\frac{\log\left(N\left(F,\delta\right)\right)}{-\log\delta}$ and $\frac{\log\left(N\left(G,\delta\right)\right)}{-\log\delta}$ oscillate between two values. Further, by ensuring that these functions oscillate with different phases (see figure \ref{figure - product graph}) we can produce strict inequalities after \eqref{limsup of sum} and \eqref{liminf of sum} and so yield strict inequality in both product formulas, that is
\[
\dim_{LB}\left(F\times G\right) < \dim_{LB}\left(F\right) + \dim_{LB}\left(G\right) < \dim_{B}\left(F\right) + \dim_{B}\left(G\right) < \dim_{B}\left(F\times G\right).
\]

\begin{figure}
   \vspace*{8pt}
\includegraphics[scale=0.5]{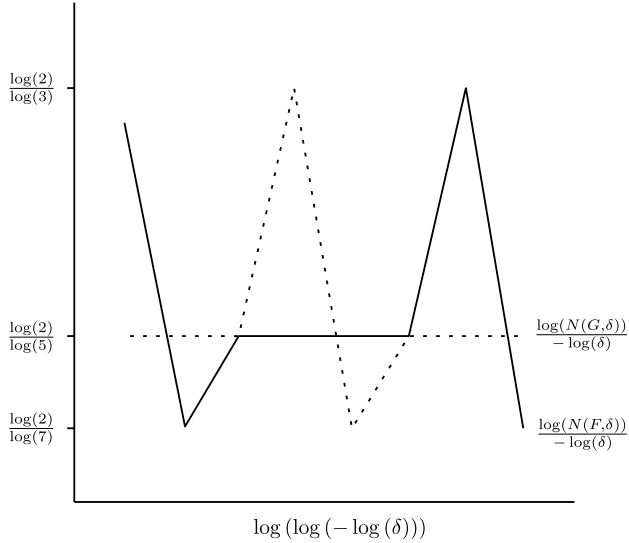}
   \vspace*{8pt}
\caption{The box-counting functions for the sets $F$ and $G$ constructed in the next section (explicitly computed here for small $\delta$) oscillate between $\frac{\log\left(2\right)}{\log\left(3\right)}$ and $\frac{\log\left(2\right)}{\log\left(7\right)}$. The $x$-axis is scaled so that the slow oscillation can be graphed, and this oscillation continues as $\log\left(\log\left(-\log\left(\delta\right)\right)\right)\rightarrow \infty$, that is as $\delta\rightarrow 0$. The differing phases guarantee that the sum of these functions doesn't approach either $\frac{\log\left(2\right)}{\log\left(3\right)}+\frac{\log\left(2\right)}{\log\left(3\right)}$ or $\frac{\log\left(2\right)}{\log\left(7\right)}+\frac{\log\left(2\right)}{\log\left(7\right)}$.}
\label{figure - product graph}
\end{figure}

To this end we construct variations of the Cantor middle-third set from the initial interval $\left[0,1\right]$ except at each stage we use one of the three generators $\gen_{3},\gen_{5}$ or $\gen_{7}$ which remove the middle $\frac{1}{3},\frac{3}{5}$ or $\frac{5}{7}$ of each interval respectively. Note that if we exclusively use the $\gen_{3}$ generator we produce the usual Cantor middle-third set, which has lower and upper box-counting dimension $\frac{\log\left(2\right)}{\log\left(3\right)}$ and if we exclusively use the $\gen_{7}$ generator we produce a similar Cantor set with lower and upper box-counting dimension $\frac{\log\left(2\right)}{\log\left(7\right)}$. By switching generators at certain stages of our construction we can cause $\frac{\log\left(N\left(F,\delta\right)\right)}{-\log\delta}$ to oscillate between these values, providing that we apply the generators a sufficiently large number of times.

\paragraph*{}
To simplify notation, we choose a sequence of integers $K_{j}=10^{2^{j}}$ which increases sufficiently quickly that
\begin{align}
\sum_{i=0}^{j} K_{i} &< K_{j+1} \label{Kj upper bound}\\
\sum_{i=0}^{j} K_{i} - K_{i-1} &>K_{j-1} \label{Kj lower bound}\\
\frac{\log\left(7\right)}{\log\left(3\right)}K_{j}&<K_{j+1}\label{Kj log3log7}
\intertext{and}
\frac{\sum_{i=0}^{j-1} K_{i}}{K_{j}} &\rightarrow 0 \quad \text{as $j\rightarrow \infty$} \label{Kj ratio limit}
\end{align}

\section{Constructing sets $F$ and $G$}
We construct two sets $F$ and $G$ using the following iterative procedure: For $F$ apply the following generator at the $j^{\text{th}}$ stage
\[
\begin{cases}
\gen_{3} & K_{6n}<j\leq K_{6n+1}\quad \text{for some} \quad n\in\mathbb{N}\\
\gen_{7} &K_{6n+1}<j\leq K_{6n+2}\quad \text{for some} \quad n\in\mathbb{N}\\
\gen_{5} & \text{otherwise},
\end{cases}
\]
and for $G$ apply the following generator at the $j^{\text{th}}$ stage
\[
\begin{cases}
\gen_{3} & K_{6m+3}<j\leq K_{6m+4}\quad \text{for some} \quad m\in\mathbb{N}\\
\gen_{7} & K_{6m+4}<j\leq K_{6m+5}\quad \text{for some} \quad m\in\mathbb{N}\\
\gen_{5} & \text{otherwise}.
\end{cases}
\]
Let $f_{3}\left(j\right)$ be the number of times $\gen_{3}$ has been applied and $f_{7}\left(j\right)$ the number of times $\gen_{7}$ has been applied in the construction $F$ by stage $j$. With this notation $\gen_{5}$ has been applied $j-f_{3}\left(j\right)-f_{7}\left(j\right)$ times by stage $j$. Similarly define $g_{3}\left(j\right)$ and $g_{7}\left(j\right)$ for the construction of $G$. Clearly these functions are non-decreasing and we refrain from writing them explicitly except to note that
\begin{equation}\label{explicit f(Kj)}
\begin{aligned}
f_{3}\left(K_{6n+1}\right)&=\sum_{i=0}^{n} K_{6i+1} - K_{6i} & 
f_{7}\left(K_{6n+2}\right)&=\sum_{i=0}^{n} K_{6i+2} - K_{6i+1}\\
g_{3}\left(K_{6m+4}\right)&=\sum_{i=0}^{m} K_{6i+4} - K_{6i+3} & 
g_{7}\left(K_{6m+5}\right)&=\sum_{i=0}^{m} K_{6i+5} - K_{6i+4}
\end{aligned}
\end{equation}
and
\begin{equation}\label{f constant on intervals}
\begin{split}
f_{3}\left(j\right)=f_{3}\left(K_{6n+1}\right) \quad\text{for}\quad K_{6n+1}<j\leq K_{6n+6}\\
f_{7}\left(j\right)=f_{7}\left(K_{6n+2}\right) \quad\text{for}\quad K_{6n+2}<j\leq K_{6n+7}\\
g_{3}\left(j\right)=g_{3}\left(K_{6m+3}\right) \quad\text{for}\quad K_{6m+3}<j\leq K_{6m+8}\\
g_{7}\left(j\right)=g_{7}\left(K_{6m+3}\right) \quad\text{for}\quad K_{6m+4}<j\leq K_{6m+9}
\end{split}
\end{equation}

Denote the sets at the $j^{th}$ stage of the construction of $F$ and $G$ by 
\begin{itemize}
\item[] $F_{j}$, which consists of $2^{j}$ intervals of length $3^{-f_{3}\left(j\right)}7^{-f_{7}\left(j\right)}5^{-j+f_{3}\left(j\right)+f_{7}\left(j\right)}$ and
\item[] $G_{j}$, which consists of $2^{j}$ intervals of length $3^{-g_{3}\left(j\right)}7^{-g_{7}\left(j\right)}5^{-j+g_{3}\left(j\right)+g_{7}\left(j\right)}$
\end{itemize}
so that $F$ and $G$ are defined by $F=\bigcap F_{j}$ and $G=\bigcap G_{j}$. Note that for every $j$ the endpoints of the intervals in $F_{j}$ and $G_{j}$ are in $F$ and $G$ respectively as the generators only remove the middle of each interval.
\begin{proposition}\label{N=M=2j prop}
For $\delta$ such that
\begin{equation}\label{length-scale for stage j}
3^{-f_{3}\left(j\right)}7^{-f_{7}\left(j\right)}5^{-j+f_{3}\left(j\right)+f_{7}\left(j\right)}\leq\delta< 3^{-f_{3}\left(j-1\right)}7^{-f_{7}\left(j-1\right)}5^{-\left(j-1\right)+f_{3}\left(j-1\right)+f_{7}\left(j-1\right)}
\end{equation}
we have $N\left(F,\delta\right)=M\left(F,\delta\right)=2^{j}$.
\end{proposition}
We refer to those $\delta$ in the range \eqref{length-scale for stage j} as length-scales corresponding to stage $j$ in the construction of $F$. Clearly every $1>\delta>0$ is a length scale corresponding to exactly one stage $j_{\delta}$ and $j_{\delta}\rightarrow\infty$ as $\delta\rightarrow 0$. We refer to length-scales corresponding to the construction of $G$ in an analogous fashion.

\begin{proof}
For $\delta$ in the range \eqref{length-scale for stage j} the obvious cover consisting of all intervals in $F_{j}$ gives $N\left(F,\delta\right)\leq 2^{j}$. The opposite inequality comes from the fact that a set with diameter $\delta$ in this range intersects at most one $\left(j-1\right)^{\text{th}}$ stage interval $I$ but cannot cover both $j^{\text{th}}$ stage subintervals of $I$ (see figure \ref{figure - 2jballsneeded}) so at least $2\times 2^{j-1}=2^{j}$ elements are needed to form a cover of $F$ .\\
Next, $\delta$ in the range \eqref{length-scale for stage j} is less than the length of the intervals in $F_{j-1}$ so balls of diameter $\delta$ centred on the end points of the intervals of $F_{j-1}$ are disjoint and have centres in $F$ (see figure \ref{figure - 2jdisjointballscover}). This gives two disjoint balls for each interval, hence $M\left(F,\delta\right)\geq 2\times 2^{j-1}=2^{j}$.
For the opposite inequality suppose for a contradiction that $M\left(F,\delta\right)> 2^{j}$, then at least one of the $2^{j-1}$ intervals in $F_{j-1}$ contains the centres of least three disjoint balls with centres in $F$. Let $I$ be such an interval. At the next stage of the construction $I$ is split into two sub-intervals, one of which contains the centres of at least two of these three disjoint balls. However, this $j^{\text{th}}$ stage subinterval has length no greater than $\delta$ so two closed balls of diameter $\delta$ centred in this interval cannot be disjoint (see figure \ref{figure - 2jdisjointballsmax}), which is a contradiction.
\end{proof}
\begin{figure}
   \vspace*{8pt}
\includegraphics[scale=0.7]{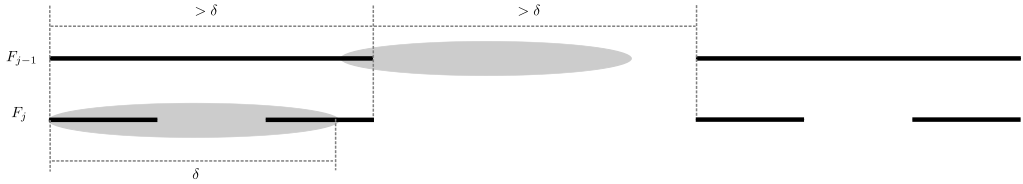}
   \vspace*{8pt}
\caption{A section of the sets $F_{j-1}$ and $F_{j}$ (shown in black) to illustrate that each set (shown as grey ellipses) with diameter $\delta<3^{-f_{3}\left(j-1\right)}7^{-f_{7}\left(j-1\right)}5^{-\left(j-1\right)+f_{3}\left(j-1\right)+f_{7}\left(j-1\right)}$ (i.e. less than the length of the intervals of $F_{j-1}$) can neither intersect two intervals of $F_{j-1}$ nor cover two intervals of $F_{j}$.}
\label{figure - 2jballsneeded}
\end{figure}
\begin{figure}
   \vspace*{8pt}
\includegraphics[scale=0.7]{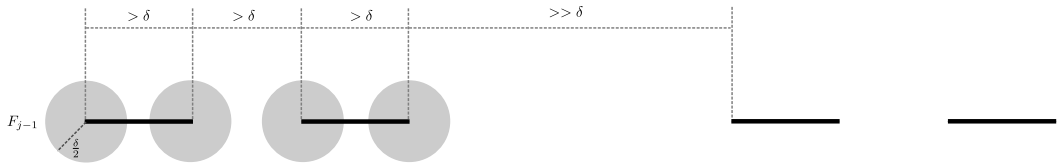}
   \vspace*{8pt}
\caption{A section of the set $F_{j-1}$ (shown in black) to illustrate that balls (shown in grey) with diameter $\delta<3^{-f_{3}\left(j-1\right)}7^{-f_{7}\left(j-1\right)}5^{-\left(j-1\right)+f_{3}\left(j-1\right)+f_{7}\left(j-1\right)}$ (i.e. less than the length of the intervals of $F_{j-1}$) with centres the endpoints of the intervals $F_{j-1}$ are disjoint, giving two disjoint balls for each interval of $F_{j-1}$.}
\label{figure - 2jdisjointballscover}
\end{figure}
\begin{figure}
   \vspace*{8pt}
\includegraphics[scale=0.7]{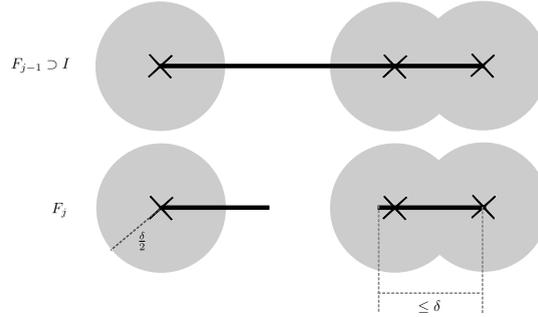}
   \vspace*{8pt}
\caption{A sub-interval $I\subset F_{j-1}$ (shown in black) which contains the centres of three balls (shown in grey) with diameter $\delta\geq 3^{-f_{3}\left(j\right)}7^{-f_{7}\left(j\right)}5^{-j+f_{3}\left(j\right)+f_{7}\left(j\right)}$ (i.e. greater than the the length of the intervals of $F_{j}$) with centres in $F$. As the centres are also contained in $F_{j}$ (also shown in black) at least one interval of $F_{j}$ contains the centres of two of these balls. Consequently, the distance between the centres is at most the length of an interval of $F_{j}$ which is at most $\delta$ which is the sum of the radii of the closed balls so the two balls are not disjoint.}
\label{figure - 2jdisjointballsmax}
\end{figure}
Adapting the argument we can prove a similar proposition for the set $G$.
\begin{proposition}\label{N=M=2j prop G}
For $\delta$ such that
\begin{equation}
3^{-g_{3}\left(j\right)}7^{-g_{7}\left(j\right)}5^{-j+g_{3}\left(j\right)+g_{7}\left(j\right)}\leq\delta< 3^{-g_{3}\left(j-1\right)}7^{-g_{7}\left(j-1\right)}5^{-\left(j-1\right)+g_{3}\left(j-1\right)+g_{7}\left(j-1\right)}
\end{equation}
we have $N\left(G,\delta\right)=M\left(G,\delta\right)=2^{j}$.
\end{proposition}

By taking logarithms we obtain the following more useful form of propositions \ref{N=M=2j prop} and \ref{N=M=2j prop G}
\begin{multline}\label{delta implies N(F,delta)}
f_{3}\left(j-1\right)\left[\log\left(3\right)-\log\left(5\right)\right]+f_{7}\left(j-1\right)\left[\log\left(7\right)-\log\left(5\right)\right] + \left(j-1\right)\log\left(5\right)\\
< -\log\delta \leq f_{3}\left(j\right)\left[\log\left(3\right)-\log\left(5\right)\right]+f_{7}\left(j\right)\left[\log\left(7\right)-\log\left(5\right)\right] + j\log\left(5\right)\\
\Rightarrow \log\left(N\left(F,\delta\right)\right) = \log\left(M\left(F,\delta\right)\right) = j\log\left(2\right)
\end{multline}
and
\begin{multline}\label{delta implies N(G,delta)}
g_{3}\left(j-1\right)\left[\log\left(3\right)-\log\left(5\right)\right]+g_{7}\left(j-1\right)\left[\log\left(7\right)-\log\left(5\right)\right] + \left(j-1\right)\log\left(5\right)\\
< -\log\delta \leq g_{3}\left(j\right)\left[\log\left(3\right)-\log\left(5\right)\right]+g_{7}\left(j\right)\left[\log\left(7\right)-\log\left(5\right)\right] + j\log\left(5\right)\\
\Rightarrow \log\left(N\left(G,\delta\right)\right) = \log\left(M\left(G,\delta\right)\right) = j\log\left(2\right)
\end{multline}

\paragraph*{}
The essential feature of our construction is that the sets $F$ and $G$ at some length-scales look like the Cantor middle-third set, while at other length-scales look like the Cantor middle-$\frac{5}{7}^{\text{th}}$ set. Further, this `local' behaviour is maintained over sufficient length-scales that the box-counting limits of $F$ and $G$ at these length-scales approach the box-counting dimensions of the relevant Cantor set, which we will establish in the following section. We conclude this section by proving that at any length-scale the sets $F$ and $G$ do not both look like the middle-third set nor do they both look like the middle-$\frac{5}{7}^{\text{th}}$ set.

\begin{lemma}\label{delta not in both length-scales upper}
No $\delta$ is both a length-scale corresponding to some stage in $\left(K_{6n},K_{6n+2}\right]$ in the construction of $F$ for some $n\in\mathbb{N}$ and a length-scale corresponding to some stage in $\left(K_{6m+3},K_{6m+5}\right]$ in the construction of $G$ for some $m\in\mathbb{N}$.
\end{lemma}
\begin{proof}
Assume for a contradiction that $\delta$ is such a length-scale, that is
\begin{multline}\label{fg inequality 1}
3^{-f_{3}\left(K_{6n+2}\right)}7^{-f_{7}\left(K_{6n+2}\right)}5^{-K_{6n+2}+f_{3}\left(K_{6n+2}\right)+f_{7}\left(K_{6n+2}\right)}\\ \leq \delta < 3^{-f_{3}\left(K_{6n}\right)}7^{-f_{7}\left(K_{6n}\right)}5^{-K_{6n}+f_{3}\left(K_{6n}\right)+f_{7}\left(K_{6n}\right)}
\end{multline}
and
\begin{multline}\label{fg inequality 2}
3^{-g_{3}\left(K_{6m+5}\right)}7^{-g_{7}\left(K_{6m+5}\right)}5^{-K_{6m+5}+g_{3}\left(K_{6m+5}\right)+g_{7}\left(K_{6m+5}\right)}\\ \leq \delta < 3^{-g_{3}\left(K_{6m+3}\right)}7^{-g_{7}\left(K_{6m+3}\right)}5^{-K_{6m+3}+g_{3}\left(K_{6m+3}\right)+g_{7}\left(K_{6m+3}\right)}
\end{multline}
We first demonstrate that this could only hold if $n=m$. From \eqref{fg inequality 1} and \eqref{fg inequality 2} we have $7^{-K_{6n+2}} \leq \delta < 3^{-K_{6n}}$ and $7^{-K_{6m+5}} \leq \delta < 3^{-K_{6m}}$ respectively, which in turn yield $7^{-K_{6n+2}} < 3^{-K_{6m}}$ and $7^{-K_{6m+5}} < 3^{-K_{6n}}$. Taking logarithms we get
\[
K_{6m} < K_{6n+2}\frac{\log\left(7\right)}{\log\left(3\right)}< K_{6n+3}\quad\text{and}\quad K_{6n} < K_{6m+5}\frac{\log\left(7\right)}{\log\left(3\right)}< K_{6m+6}
\]
where the final inequalities follow from the growth property \eqref{Kj log3log7}. We conclude that $6m< 6n+3$ and $6n< 6m+6$, which implies that $m=n$. With this restriction, if $\delta$ is such a length-scale then the lower bound from \eqref{fg inequality 1} and the upper bound from \eqref{fg inequality 2} imply
\begin{multline*}
3^{-f_{3}\left(K_{6n+2}\right)}7^{-f_{7}\left(K_{6n+2}\right)}5^{-K_{6n+2}+f_{3}\left(K_{6n+2}\right)+f_{7}\left(K_{6n+2}\right)} \\ 
< 3^{-g_{3}\left(K_{6n+3}\right)}7^{-g_{7}\left(K_{6n+3}\right)}5^{-K_{6n+3}+g_{3}\left(K_{6n+3}\right)+g_{7}\left(K_{6n+3}\right)}
\end{multline*}
For clarity we suppress the argument of the functions and take logarithms which yields
\begin{equation}\label{supressed argument condition}
\left[\log\left(5\right)-\log\left(3\right)\right]\left(f_{3}-g_{3}\right) + \left[\log\left(5\right)-\log\left(7\right)\right]\left(f_{7}-g_{7}\right) + \log\left(5\right)\left(K_{6n+3}-K_{6n+2}\right) < 0.
\end{equation}
The bounds in \eqref{Kj upper bound} and \eqref{Kj lower bound} give
\begin{align}
g_{3}\left(K_{6n+3}\right)-f_{3}\left(K_{6n+2}\right) &= \sum_{i=0}^{n-1} K_{6i+4} - K_{6i+3} - \sum_{i=0}^{n} K_{6i+1}-K_{6i}\notag\\
&< K_{6n-1} - K_{6n} < 0 \label{g3-f3}
\intertext{and}
f_{7}\left(K_{6n+2}\right)-g_{7}\left(K_{6n+3}\right) &= \sum_{i=0}^{n} K_{6i+2} - K_{6i+1} - \sum_{i=0}^{n-1} K_{6i+5} - K_{6i+4}\notag\\
&< K_{6n+3} - K_{6n-2}. \label{f7-g7}
\end{align}
Consequently we drop the positive first term of \eqref{supressed argument condition} which implies
\begin{align*}
\left[\log\left(5\right)-\log\left(7\right)\right]\left(K_{6n+3} - K_{6n-2}\right) + \log\left(5\right)\left(K_{6n+3}-K_{6n+2}\right) &< 0
\intertext{which is that}
\left[2\log\left(5\right)-\log\left(7\right)\right]K_{6n+3} + \left[\log\left(7\right)-\log\left(5\right)\right]K_{6n-2} - \log\left(5\right)K_{6n+2} &< 0
\end{align*}
After dropping the positive middle term and rearranging we get
\begin{align*}
K_{6n+3}&< \frac{\log\left(5\right)}{\left[2\log\left(5\right)-\log\left(7\right)\right]}K_{6n+2}
\intertext{so that}
K_{6n+3}&< \frac{\log\left(7\right)}{\log\left(3\right)}K_{6n+2}
\end{align*}
however, by condition \eqref{Kj log3log7} the right hand side is less than $K_{6n+3}$ giving the required contradiction.
\end{proof}
Consequently, at any length-scale the sets $F$ and $G$ do not both look like the Cantor middle-third set. This lemma gives the following useful corollary:

\begin{corollary}\label{subsequence corollary upper}
Every sequence $\set{\delta_{i}}$ with $\delta_{i}\rightarrow 0$ either contains a subsequence $\set{\delta_{i_{n}}}$ with each $\delta_{i_{n}}$ corresponding to some stage $j_{\delta_{i_{n}}}\in \left(K_{6n+2},K_{6n+6}\right]$ in the construction of $F$ or contains a subsequence $\set{\delta_{i_{m}}}$ corresponding to some stage $j_{\delta_{i_{m}}}\in \left(K_{6m+5},K_{6m+9}\right]$ in the construction of $G$.
\end{corollary}
\begin{proof}
If the sequence $\set{\delta_{i}}$ did not contain such a subsequence, then there is a $\Delta>0$ such that each $\delta_{i}<\Delta$ is neither a length-scale corresponding to some stage  $j\in\left(K_{6n+2},K_{6n+6}\right]$ in the construction of $F$ nor a length-scale corresponding to some stage $j\in \left(K_{6m+5},K_{6m+9}\right]$ in the construction of $G$. Consequently, $\delta_{i}$ is a length scale corresponding to a stage $j\in\left(K_{6n},K_{6n+2}\right]$ in the construction of $F$ and also a length-scale corresponding to a stage $j\in\left(K_{6m+3},K_{6m+5}\right]$ which, from lemma \ref{delta not in both length-scales upper}, is contradictory.
\end{proof}

The corresponding lemma that $F$ and $G$ do not both look like the Cantor middle-$\frac{5}{7}^{\text{th}}$ set at any length-scale and the corresponding corollary for subsequences are proved in a similar way.

\begin{lemma}\label{delta not in both length-scales lower}
No $\delta$ is both a length-scale corresponding to some stage in $\left(K_{6n+1},K_{6n+3}\right]$ in the construction of $F$ for some $n\in\mathbb{N}$ and a length-scale corresponding to some stage in $\left(K_{6m+4},K_{6m+6}\right]$ in the construction of $G$ for some $m\in\mathbb{N}$,
\end{lemma}

\begin{corollary}\label{subsequence corollary lower}
Every sequence $\set{\delta_{i}}$ with $\delta_{i}\rightarrow 0$ either contains a subsequence $\set{\delta_{i_{n}}}$ with each $\delta_{i_{n}}$ corresponding to some stage $j_{\delta_{i_{n}}}\in \left(K_{6n+3},K_{6n+7}\right]$ in the construction of $F$ or contains a subsequence $\set{\delta_{i_{m}}}$ corresponding to some stage $j_{\delta_{i_{m}}}\in \left(K_{6m},K_{6m+4}\right]$ in the construction of $G$.
\end{corollary}

\section{Calculating box-counting dimensions}
In order to establish the box-counting dimensions of $F$ and $G$ we need the following proposition on the behaviour of the generator-counting functions at the limit:

\begin{proposition}\label{f Kj ratio limit}
\[
\frac{f_{3}\left(K_{6n+l}\right)}{K_{6n+l}} \rightarrow
\begin{cases}
1 & l=1\\
0 & 0\leq l < 6,\; l\neq 1 
\end{cases}\quad
\frac{f_{7}\left(K_{6n+l}\right)}{K_{6n+l}} \rightarrow
\begin{cases}
1 & l=2\\
0 & 0\leq l < 6,\; l\neq 2 
\end{cases}
\]
\[
\frac{g_{3}\left(K_{6m+l}\right)}{K_{6m+l}} \rightarrow
\begin{cases}
1 & l=4\\
0 & 0\leq l < 6,\; l\neq 4 
\end{cases}\quad
\frac{g_{7}\left(K_{6m+l}\right)}{K_{6m+l}} \rightarrow
\begin{cases}
1 & l=5\\
0 & 0\leq l < 6,\; l\neq 5 
\end{cases}
\]
as $n,m\rightarrow \infty$.
\end{proposition}
\begin{proof}
From \eqref{explicit f(Kj)} we have
\[
\frac{f_{3}\left(K_{6n+1}\right)}{K_{6n+1}} = \frac{\sum_{i=0}^{n} K_{6i+1} - K_{6i}}{K_{6n+1}} = 1 + \frac{\sum_{i=0}^{n-1} K_{6i+1}}{K_{6n+1}} - \frac{\sum_{i=0}^{n} K_{6i}}{K_{6n+1}}
\]
which converges to 1 as $n\rightarrow \infty$ by \eqref{Kj ratio limit}.
By \eqref{f constant on intervals} we have $f_{3}\left(K_{6n+l}\right)=f_{3}\left(K_{6n+1}\right)$ for $2\leq l<6$ so
\[
\frac{f_{3}\left(K_{6n+l}\right)}{K_{6n+l}} = \frac{f_{3}\left(K_{6n+1}\right)}{K_{6n+l}} = \frac{\sum_{i=0}^{n} K_{6i+1} - K_{6i}}{K_{6n+l}} 
\]
which converges to 0 as $n\rightarrow \infty$ by \eqref{Kj ratio limit}. Similarly, $f_{3}\left(K_{6n+0}\right) = f_{3}\left(K_{6\left(n-1\right)+1}\right)$ so
\[
\frac{f_{3}\left(K_{6n+0}\right)}{K_{6n+0}} = \frac{f_{3}\left(K_{6\left(n-1\right)+1}\right)}{K_{6n}} = \frac{\sum_{i=0}^{n-1} K_{6i+1} - K_{6i}}{K_{6n}} \rightarrow 0
\]
The remaining results follow analogously.
\end{proof}

\begin{lemma}\label{upper box lemma}
$\dim_{B}\left(F\right)=\dim_{B}\left(G\right)=\frac{\log\left(2\right)}{\log\left(3\right)}.$
\end{lemma}
\begin{proof}
First, we show that $\dim_{B}\left(F\right)\leq \frac{\log\left(2\right)}{\log\left(3\right)}$. Writing $j_{\delta}$ for the stage associated with the length-scale $\delta$ we have from \eqref{delta implies N(F,delta)}
\begin{align*}
\frac{\log\left(N\left(F,\delta\right)\right)}{-\log \delta}&< \frac{j_{\delta}\log\left(2\right)}{f_{3}\left(j_{\delta}-1\right)\left[\log\left(3\right)-\log\left(5\right)\right] + f_{7}\left(j_{\delta}-1\right)\left[\log\left(7\right)-\log\left(5\right)\right] + \left(j_{\delta}-1\right)\log\left(5\right)}\\
&\leq \frac{j_{\delta}\log\left(2\right)}{f_{3}\left(j_{\delta}-1\right)\left[\log\left(3\right)-\log\left(5\right)\right] + \left(j_{\delta}-1\right)\log\left(5\right)}
\intertext{and as $f_{3}\left(j_{\delta}-1\right)< j_{\delta}-1$}
&< \frac{j_{\delta}\log\left(2\right)}{\left(j_{\delta}-1\right)\left[\log\left(3\right)-\log\left(5\right)\right] + \left(j_{\delta}-1\right)\log\left(5\right)} = \frac{j_{\delta}\log\left(2\right)}{\left(j_{\delta}-1\right)\log\left(3\right)}\\
\end{align*}
which converges to $\frac{\log\left(2\right)}{\log\left(3\right)}$ as $\delta\rightarrow 0$. Similarly we can show $\dim_{B}\left(G\right)\leq\frac{\log\left(2\right)}{\log\left(3\right)}$

\paragraph*{}
Next, if we take the sequence $\set{\delta_{n}}$ with
\[
-\log\delta_{n} = f_{3}\left(K_{6n+1}\right)\left[\log\left(3\right)-\log\left(5\right)\right]+ f_{7}\left(K_{6n+1}\right)\left[\log\left(7\right)-\log\left(5\right)\right]+ K_{6n+1}\log\left(5\right)
\]
we have from \eqref{delta implies N(F,delta)} that $\log\left(N\left(F,\delta_{n}\right)\right)= K_{6n+1}\log\left(2\right)$. Consequently,
\begin{align*}
\frac{\log\left(N\left(F,\delta_{n}\right)\right)}{-\log \delta_{n}} &= \frac{K_{6n+1}\log\left(2\right)}{f_{3}\left(K_{6n+1}\right)\left[\log\left(3\right)-\log\left(5\right)\right]+ f_{7}\left(K_{6n+1}\right)\left[\log\left(7\right)-\log\left(5\right)\right]+ K_{6n+1}\log\left(5\right)}\\
&= \frac{\log\left(2\right)}{\frac{f_{3}\left(K_{6n+1}\right)}{K_{6n+1}}\left[\log\left(3\right)-\log\left(5\right)\right]+\frac{f_{7}\left(K_{6n+1}\right)}{K_{6n+1}}\left[\log\left(7\right)-\log\left(5\right)\right]+ \log\left(5\right)}\\
&\rightarrow \frac{\log\left(2\right)}{\log\left(3\right)-\log\left(5\right) + 0 + \log\left(5\right)} = \frac{\log\left(2\right)}{\log\left(3\right)}
\end{align*}
as $n\rightarrow\infty$ by the convergence results \eqref{f Kj ratio limit} so that $\dim_{B}\left(F\right)\geq \frac{\log\left(2\right)}{\log\left(3\right)}$\\
A similar sequence gives the corresponding inequality for $G$.
\end{proof}

\begin{lemma}\label{lower box lemma}
$\dim_{LB}\left(F\right)=\dim_{LB}\left(G\right)=\frac{\log\left(2\right)}{\log\left(7\right)}.$
\end{lemma}
\begin{proof}
For all $\delta>0$ the implication \eqref{delta implies N(F,delta)} gives
\begin{align*}
\frac{\log\left(N\left(F,\delta\right)\right)}{-\log \delta} &\geq \frac{j_{\delta}\log\left(2\right)}{ f_{3}\left(j_{\delta}\right)\left[\log\left(3\right)-\log\left(5\right)\right]+ f_{7}\left(j_{\delta}\right)\left[\log\left(7\right)-\log\left(5\right)\right]+j_{\delta}\log\left(5\right)}\\
&\geq \frac{j_{\delta}\log\left(2\right)}{ f_{7}\left(j_{\delta}\right)\left[\log\left(7\right)-\log\left(5\right)\right]+j_{\delta}\log\left(5\right)}\\
&> \frac{j_{\delta}\log\left(2\right)}{ j_{\delta}\left[\log\left(7\right)-\log\left(5\right)\right]+j_{\delta}
\log\left(5\right)}=\frac{\log\left(2\right)}{\log\left(7\right)}
\end{align*}
Next, if we take the sequence $\set{\delta_{n}}$ with 
\[
-\log\delta_{n} = f_{3}\left(K_{6n+2}\right)\left[\log\left(3\right)-\log\left(5\right)\right]+ f_{7}\left(K_{6n+2}\right)\left[\log\left(7\right)-\log\left(5\right)\right]+ K_{6n+2}\log\left(5\right)
\]
we have
\begin{align*}
\frac{\log\left(N\left(F,\delta_{n}\right)\right)}{-\log \delta_{n}} &= \frac{K_{6n+2}\log\left(2\right)}{ f_{3}\left(K_{6n+2}\right)\left[\log\left(3\right)-\log\left(5\right)\right]+ f_{7}\left(K_{6n+2}\right)\left[\log\left(7\right)-\log\left(5\right)\right]+ K_{6n+2}\log\left(5\right)}\\
&= \frac{\log\left(2\right)}{\frac{f_{3}\left(K_{6n+2}\right)}{K_{6n+2}}\left[\log\left(3\right)-\log\left(5\right)\right]+\frac{f_{7}\left(K_{6n+2}\right)}{K_{6n+2}}\left[\log\left(7\right)-\log\left(5\right)\right]+ \log\left(5\right)}\\
&\rightarrow \frac{\log\left(2\right)}{0 + \log\left(7\right)-\log\left(5\right) + \log\left(5\right)} = \frac{\log\left(2\right)}{\log\left(7\right)}
\end{align*}
as $n\rightarrow\infty$ by the convergence results \eqref{f Kj ratio limit}. Hence $\dim_{LB}\left(F\right)=\frac{\log\left(2\right)}{\log\left(7\right)}$, and similarly $\dim_{LB}\left(G\right)=\frac{\log\left(2\right)}{\log\left(7\right)}$.
\end{proof}
Consequently, both $F$ and $G$ have unequal upper and lower box-counting dimensions:
\begin{corollary}
$\dim_{LB}\left(F\right)=\dim_{LB}\left(G\right) < \dim_{B}\left(F\right)=\dim_{B}\left(G\right)$.
\end{corollary}

\paragraph*{}
Whilst the above lemmas demonstrate that for $F$ and $G$ there are sequences of length-scales $\set{\delta_{n}}$ with $\lim_{n\rightarrow\infty}\frac{\log\left(N\left(F,\delta_{n}\right)\right)}{-\log\delta_{n}}$ equal to $\frac{\log\left(2\right)}{\log\left(3\right)}$ or equal to $\frac{\log\left(2\right)}{\log\left(7\right)}$ we now show that for a large class of sequences (in fact the very sequences that corollaries \ref{subsequence corollary upper} and \ref{subsequence corollary lower} produce) this limit, if it exists, is bounded by $\frac{\log\left(2\right)}{\log\left(5\right)}$.

\begin{lemma}\label{deltan N(F,deltan) converge to log2log5}
Suppose $\set{\delta_{n}}$ is a sequence such that each length-scale $\delta_{n}$ corresponds to the construction of $F$ at some stage $j_{n}\in \left(K_{6n+2},K_{6n+6}\right]$, then if the limit exists
\[\lim_{n\rightarrow\infty}\frac{\log\left(N\left(F,\delta_{n}\right)\right)}{-\log\delta_{n}}\leq\frac{\log\left(2\right)}{\log\left(5\right)}.
\]
\end{lemma}
Essentially, these stages are sufficiently far from the range $\left(K_{6n},K_{6n+1}\right]$ where $\gen_{3}$ is applied so that the set $F$ does not look like the Cantor middle-third set at these stages. The proof relies on the fact that by stage $j_{n}$ the generator $\gen_{3}$ has not been applied for at least the last $K_{6n+2}-K_{6n+1}$ stages.
\begin{proof}
For each $n\in\mathbb{N}$ from \eqref{delta implies N(F,delta)}
\begin{align*}
\frac{\log\left(N\left(F,\delta_{n}\right)\right)}{-\log\delta_{n}} &\leq \frac{j_{n}\log\left(2\right)}{f_{3}\left(j_{n}-1\right)\left[\log\left(3\right)-\log\left(5\right)\right]+f_{7}\left(j_{n}-1\right)\left[\log\left(7\right)-\log\left(5\right)\right] + \left(j_{n}-1\right)\log\left(5\right)}\\
&\leq \frac{j_{n}\log\left(2\right)}{f_{3}\left(j_{n}-1\right)\left[\log\left(3\right)-\log\left(5\right)\right]+ \left(j_{n}-1\right)\log\left(5\right)}
\intertext{From \eqref{f constant on intervals} we have $f_{3}\left(j_{n}-1\right)=f_{3}\left(K_{6n+1}\right)$ so that}
&=\frac{j_{n}\log\left(2\right)}{f_{3}\left(K_{6n+1}\right)\left[\log\left(3\right)-\log\left(5\right)\right]+ \left(j_{n}-1\right)\log\left(5\right)}\\
&=\frac{\log\left(2\right)}{\frac{f\left(K_{6n+1}\right)}{j_{n}}\left[\log\left(3\right)-\log\left(5\right)\right]+\log\left(5\right)-\frac{1}{j_{n}}\log\left(5\right)}.
\intertext{Next, as $j_{n}>K_{6n+2}$}
&\leq \frac{\log\left(2\right)}{\frac{f_{3}\left(K_{6n+1}\right)}{K_{6n+2}}\left[\log\left(3\right)-\log\left(5\right)\right]+\log\left(5\right)-\frac{1}{j_{n}}\log\left(5\right)}\\
&\rightarrow \frac{\log\left(2\right)}{\log\left(5\right)}
\end{align*}
as $n\rightarrow\infty$ by the convergence result \eqref{f Kj ratio limit}.
\end{proof}
The corresponding result for $G$, proved in a similar way, is as follows.
\begin{lemma}\label{deltan N(G,deltan) converge to log2log5}
Suppose $\set{\delta_{m}}$ is a sequence such that each length-scale $\delta_{m}$ corresponds to the construction of $G$ at some stage $j_{m}\in \left(K_{6m+5},K_{6m+9}\right]$, then if the limit exists
\[\lim_{m\rightarrow\infty}\frac{\log\left(N\left(G,\delta_{m}\right)\right)}{-\log\delta_{m}}\leq\frac{\log\left(2\right)}{\log\left(5\right)}.
\]
\end{lemma}
The following results for lower bounds are also proved similarly.
\begin{lemma}\label{deltan M(F,deltan) converge to log2log5}
Suppose $\set{\delta_{n}}$ is a sequence such that each length-scale $\delta_{n}$ corresponds to the construction of $F$ at some stage $j_{n}\in \left(K_{6n+3},K_{6n+7}\right]$, then if the limit exists
\[\lim_{n\rightarrow\infty}\frac{\log\left(M\left(F,\delta_{n}\right)\right)}{-\log\delta_{n}}\geq\frac{\log\left(2\right)}{\log\left(5\right)}
\]
\end{lemma}
and
\begin{lemma}\label{deltan M(G,deltan) converge to log2log5}
Suppose $\set{\delta_{m}}$ is a sequence such that each length-scale $\delta_{m}$ corresponds to the construction of $G$ at some stage $j_{m}\in \left(K_{6m},K_{6m+4}\right]$, then if the limit exists
\[\lim_{m\rightarrow\infty}\frac{\log\left(M\left(G,\delta_{m}\right)\right)}{-\log\delta_{m}}\geq\frac{\log\left(2\right)}{\log\left(5\right)}.
\]
\end{lemma}

Finally, we find a bound on the box-counting dimensions of the product $F\times G$.

\begin{theorem}\label{strict upper box product inequality}
$\dim_{B}\left(F\times G\right)\leq \frac{\log\left(2\right)}{\log\left(3\right)} + \frac{\log\left(2\right)}{\log\left(5\right)}.$
\end{theorem}
\begin{proof}
We have from \eqref{limsup of sum} that
\[
\limsup_{\delta\rightarrow 0} \frac{\log\left(N\left(F\times G,\delta\right)\right)}{-\log \delta} \leq \limsup_{\delta\rightarrow 0}\left[\frac{\log\left(N\left(F,\delta\right)\right)}{-\log\delta} + \frac{\log\left(N\left(G,\delta\right)\right)}{-\log\delta}\right]
\]
so it is sufficient to show that the right hand side is no greater than $\frac{\log\left(2\right)}{\log\left(3\right)} + \frac{\log\left(2\right)}{\log\left(5\right)}$. Suppose that $\set{\delta_{i}}$ is a sequence with $\delta_{i}\rightarrow 0$ such that the limits $\lim_{i\rightarrow \infty}\frac{\log\left(N\left(F,\delta_{i}\right)\right)}{-\log \delta_{i}}$ and $\lim_{i\rightarrow \infty}\frac{\log\left(N\left(G,\delta_{i}\right)\right)}{-\log \delta_{i}}$ exist. Corollary \ref{subsequence corollary upper} guarantees that this sequence either contains a subsequence $\set{\delta_{i_{n}}}$ satisfying the hypothesis of lemma \ref{deltan N(F,deltan) converge to log2log5} or contains a subsequence $\set{\delta_{i_{m}}}$ satisfying the hypothesis of lemma \ref{deltan N(G,deltan) converge to log2log5} so at least one of $\frac{\log\left(N\left(F,\delta_{i_{n}}\right)\right)}{-\log \delta_{i_{n}}}$ and $\frac{\log\left(N\left(G,\delta_{i_{n}}\right)\right)}{-\log \delta_{i_{n}}}$ converges to $\frac{\log\left(2\right)}{\log\left(5\right)}$. Using the upper box-counting dimension from lemma \ref{upper box lemma} to bind the other term yields
\[
\lim_{n\rightarrow \infty} \frac{\log\left(N\left(F,\delta_{i_{n}}\right)\right)}{-\log \delta_{i_{n}}} + \frac{\log\left(N\left(G,\delta_{i_{n}}\right)\right)}{-\log \delta_{i_{n}}} \leq \frac{\log\left(2\right)}{\log\left(3\right)} + \frac{\log\left(2\right)}{\log\left(5\right)}
\]
a bound which also hold for the original sequence $\set{\delta_{i}}$.
As $\set{\delta_{i}}$ was an arbitrary convergent sequence,
\[
\limsup_{\delta\rightarrow 0}\left[\frac{\log\left(N\left(F,\delta\right)\right)}{-\log\delta} + \frac{\log\left(N\left(G,\delta\right)\right)}{-\log\delta}\right] \leq \frac{\log\left(2\right)}{\log\left(3\right)} + \frac{\log\left(2\right)}{\log\left(5\right)}
\]
\end{proof}

\begin{corollary}
$\dim_{B}\left(F\times G\right)< \dim_{B}\left(F\right) + \dim_{B}\left(G\right)$
\end{corollary}

\begin{theorem}
$\dim_{LB}\left(F\times G\right)\geq \frac{\log\left(2\right)}{\log\left(7\right)} + \frac{\log\left(2\right)}{\log\left(5\right)}.$
\end{theorem}
\begin{proof}
From \eqref{liminf of sum} we have
\[
\liminf_{\delta\rightarrow 0} \frac{\log\left(M\left(F\times G,\delta\right)\right)}{-\log\delta} \geq \liminf_{\delta\rightarrow 0} \left[\frac{\log\left(M\left(F,\delta\right)\right)}{-\log\delta} + \frac{\log\left(M\left(G,\delta\right)\right)}{-\log\delta} \right]
\]
so it is sufficient to prove that the right hand side is no less than $\frac{\log\left(2\right)}{\log\left(7\right)} + \frac{\log\left(2\right)}{\log\left(5\right)}$. Suppose that $\set{\delta_{i}}$ is a sequence with $\delta_{i}\rightarrow 0$ such that the limits $\lim_{i\rightarrow \infty}\frac{\log\left(N\left(F,\delta_{i}\right)\right)}{-\log \delta_{i}}$ and $\lim_{i\rightarrow \infty}\frac{\log\left(N\left(G,\delta_{i}\right)\right)}{-\log \delta_{i}}$ exist. In a similar fashion to theorem \ref{strict upper box product inequality}, corollary \ref{subsequence corollary lower} and lemmas \ref{deltan M(F,deltan) converge to log2log5} and \ref{deltan M(G,deltan) converge to log2log5} guarantee that at least one of  $\lim_{i\rightarrow\infty}\frac{\log\left(M\left(F,\delta_{i_{n}}\right)\right)}{-\log\delta_{i_{n}}}$ and $\lim_{i\rightarrow\infty}\frac{\log\left(M\left(G,\delta_{i_{n}}\right)\right)}{-\log\delta_{i_{n}}}$ is no less than $\frac{\log\left(2\right)}{\log\left(5\right)}$ so
\[
\liminf_{\delta\rightarrow 0} \left[\frac{\log\left(M\left(F,\delta\right)\right)}{-\log\delta} + \frac{\log\left(M\left(G,\delta\right)\right)}{-\log\delta}\right] \geq \frac{\log\left(2\right)}{\log\left(5\right)} + \frac{\log\left(2\right)}{\log\left(7\right)}
\]
\end{proof}

\begin{corollary}
$\dim_{LB}\left(F\times G\right) > \dim_{LB}\left(F\right) + \dim_{LB}\left(G\right)$
\end{corollary}

\begin{acknowledgements}\label{ackref}
I am greatly indebted to my PhD supervisor, James Robinson, for highlighting that there was no example of a strict product inequality in the box-counting literature and to the generosity of Isabelle Harding and Matt Gibson in providing accommodation whilst much of this work was completed.
\end{acknowledgements}

\affiliationone{
   Nick Sharples,\\
   Mathematics Institute,\\
    University of Warwick,\\
    Coventry\\
   CV4 7AL\\
   UK
   \email{n.sharples@warwick.ac.uk}}
%
\end{document}